\documentclass[11pt]{amsart}

\usepackage{amssymb,amsmath,enumerate,epsfig}%,epsf}%t1enc
\usepackage{amsfonts,color}
\usepackage{a4,latexsym}
\usepackage{parskip}
\usepackage[all]{xy}
\usepackage{multirow}
\usepackage{hyperref}
\hypersetup{pdftitle=Enriques}

\newcommand{\C} {\mathbb{C}}
\newcommand{\Q} {\mathbb{Q}}
\newcommand{\N}  {\mathbb{N}}

\newcommand{\Z}{\mathbb{Z}}

\newcommand{\PP}{\mathbb{P}}
\newcommand{\NS}{\mathop{\rm NS}}

\newcommand{\MW}{\mathop{\rm MW}}

\newcommand{\Km}{\mathop{\rm Km}}

\newtheorem{Theorem}{Theorem}%[section]
\newtheorem{Proposition}[Theorem]{Proposition}
\newtheorem{Lemma}[Theorem]{Lemma}

\theoremstyle{remark}
\newtheorem{Remark}[Theorem]{Remark}

\newtheorem{Example}[Theorem]{Example}

\theoremstyle{definition}

\begin{document}

\title{Sandwich theorems for Shioda--Inose structures}

\author{Matthias Sch\"utt}
\address{Institut f\"ur Algebraische Geometrie, Leibniz Universit\"at
  Hannover, Welfengarten 1, 30167 Hannover, Germany}
\email{schuett@math.uni-hannover.de}
\urladdr{http://www.iag.uni-hannover.de/$\sim$schuett/}

%
%
%\thanks{Partial funding from DFG grant Hu 337/6-1 is gratefully acknowledged}
%
%\dedicatory{Dedicated to Tetsuji Shioda on the occasion of his 70th birthday}
%

\date{February 22, 2011}

 \begin{abstract}
We give a geometric construction of the three infinite series of K3 surfaces
which are
sandwiched by Kummer surfaces
within a Shioda--Inose structure.
Explicit examples are also provided.
 \end{abstract}
 
 \maketitle

\section{Introduction}

Shioda--Inose structures have recently featured very prominently in the arithmetic and geometry of K3 surfaces,
relating specific K3 surfaces to Kummer surfaces in a natural way.
In \cite{Sandwich}, Shioda showed that any jacobian elliptic K3 surface
with two singular fibres of type $II^*$ is in fact sandwiched by the Kummer surface in question
(which is indeed of product type). 
Subsequently Ma gave an abstract Hodge theoretic proof that any Shioda--Inose structure can be extended to a sandwich \cite{Ma}.
However, for the generic situation of Picard number $17$ there are only five explicit geometric examples due to Kumar \cite{Kumar}, van Geemen--Sarti \cite{vGS} and Koike \cite{Koike}.

In this paper we will develop three infinite series
of K3 surfaces of Picard number (at least) $17$
with a sandwiched Shioda--Inose structure by geometric means:
\begin{Theorem}
\label{thm}
Let $N\in\N$. Assume that one of the following three alternatives holds:
\begin{enumerate}
\item
$\forall \,p\mid N: p\equiv 1 \mod 4;$
\item
$N=\prod_i p_i$ or $N=7\prod_i p_i, \forall\, i: p_i\equiv 1,2,4 \mod 7;$
\item
$N=\prod_i p_i$ or $15\prod_i p_i$ for an odd number of primes $p_i\equiv 2,8 \mod 15$
or\\
 $N=3\prod_i p_i$ or $5\prod_i p_i$ for an even number of primes $p_i\equiv 2,8 \mod 15.$
\end{enumerate}
Then for any K3 surface $X$ with a primitive embedding 
$T_X\hookrightarrow U^2+\langle-2N\rangle$
there is an explicit geometric sandwiched Shioda--Inose structure.
\end{Theorem}
%For the precise statements, please see...
The theorem will be proved by exhibiting three distinct families of K3 surfaces in Sections \ref{s1}-\ref{s3}.
See \ref{ss1}, \ref{ss2}, \ref{ss3} for the precise arguments.
Our construction uses specialisation from 4-dimensional families of K3 surfaces
via lattice enhancements
and elliptic fibrations with 2-torsion sections.
We point out that our construction in particular allows us to realise
all five previously known examples (by inspection of the discriminants, see Section \ref{ss:SI}),
but it does not give any transcendental lattices $U^2+\langle-2N\rangle$ beyond those specified in Theorem \ref{thm} (cf.~\ref{ss:further}).

The paper is organised as follows.
The next section reviews basics on Shioda--Inose structures and sandwiches.
Each of the subsequent sections is devoted to one of the families in question.
Throughout we work over the field $\C$ of complex numbers
although the equations provided will make perfect sense over any field $k$ of characteristic different from 2.

\section{Shioda--Inose structures and sandwiching}
\label{ss:SI}

A classical example of K3 surfaces consists in Kummer surfaces:
starting from an abelian surface $A$, we consider the quotient 
by inversion with respect to the group structure.
This attains 16 rational double point singularities (type $A_1$)
which can be resolved to a K3 surface denoted by $\Km(A)$.
Thus there is a rational map of degree $2$
\[
A \dasharrow \Km(A).
\]
This is also reflected in the transcendental lattices, i.e.~the orthogonal complements $T_X$
of $\NS(X)$ inside $H^2(X,\Z)$ with respect to cup-product.
Namely the transcendental lattices are similar, i.e.~the rank is constant while the intersection
form is multiplied by 2:
\[
T_{\Km(A)} = T_A(2).
\]
From the classification point of view, a natural problem is to relate $\Km(A)$ to a K3 surface
$X$ with the original transcendental lattice 
\[
T_X=T_A.
\]
This was first achieved by Shioda and Inose in \cite{SI}
in the case of product type $A\cong E \times E'$
where $E, E'$ are elliptic curves.
Their construction (geared towards K3 surfaces with Picard number 20)
makes crucial use of jacobian elliptic fibrations on $\Km(E\times E')$. 
In fact, $X$ is shown to admit a rational map to $\Km(E\times E')$ of degree $2$.
In other words  $X$ admits a Nikulin involution (8 isolated fixed points) 
whose quotient gives rise to $\Km(E\times E')$ as its resolution.
Following Morrison \cite{Mo}, exactly this is usually required as ingredient of a Shioda--Inose structure:
 \begin{eqnarray}
 \label{eq:SI}
  \xymatrix{A \ar@{-->}[dr] && X\ar@{-->}[dl]\\
 & \Km(A)&}
 \end{eqnarray}
In the situation of non-isogenous elliptic curves,
one has $\NS(A)=U$ and $T_A=U^2$, 
where $U$ denotes the hyperbolic plane $\Z^2$ with intersection form
$\begin{pmatrix} 0 & 1\\1 & 0\end{pmatrix}$.
The generic situation is somewhat different as $A$ will have Picard number one.
If $A$ is endowed with a polarisation of degree $2N (N\in\N)$, then we have 
\[
T_A=U^2+\langle-2N\rangle.
\]
In terms of lattice theory Morrison gave a complete answer which K3 surfaces admit a Shioda--Inose structure \cite{Mo}:

\begin{Theorem}[Morrison]
\label{thm:Mo}
An algebraic K3 surface $X$  admits a Shioda--Inose structure if and only if there is a primitive embedding
\[
T_X \hookrightarrow U^2+\langle-2N\rangle \;\;\; \text{ for some } N\in\N.
\]
\end{Theorem}
An equivalent criterion is that $X$ admits a (Nikulin) involution interchanging
two orthogonal copies of $E_8$ in $\NS(X)$, the unique unimodular even negative-definite lattice of rank 8.
Or even more abstractly: $E_8^2\hookrightarrow\NS(X)$.

In order to discuss sandwiching,
we return to the product type situation $A=E\times E'$ from \cite{SI}
that we alluded to above.
In \cite{Sandwich} Shioda noticed that this case comes automatically with a sandwich.
Namely $\Km(E\times E')$ itself possesses a Nikulin involution which gives rise to $X$,
thus extending the diagram \eqref{eq:SI} as follows:
 \begin{eqnarray}
 \label{eq:sw}
  \xymatrix{
  &&& \Km(E\times E') \ar@{-->}[dl]\\
  E\times E' \ar@{-->}[dr] && X\ar@{-->}[dl]&\\
 & \Km(E\times E')&&}
 \end{eqnarray}
This brings us to the problem whether every Shioda--Inose structure can be extended to a sandwich.
An affirmative answer was given recently by Ma in \cite{Ma}:

\begin{Theorem}[Ma]
Any Shioda--Inose structure admits a sandwich.
\end{Theorem}

We emphasise that Ma's proof is Hodge theoretic in nature;
in particular it does not give any geometric information.
Indeed there are only five $N\in\N$ in terms of Theorem \ref{thm:Mo} for which an explicit geometric
construction has been exhibited:

\begin{tabular}{lll}
$N=1$ &
due to Kumar \cite{Kumar} & -- Kummer surfaces\\% are generically\\
&& \;\; of jacobians of genus 2 curves \\
$N=2$ & van Geemen--Sarti \cite{vGS} &\multirow{2}{*}{-- elliptic K3 surfaces with MW rank zero}\\
$N=3,5,7$ & 
due to Koike \cite{Koike} &
%-- elliptic K3 surfaces with MW rank zero
\\
\end{tabular}

In each case the quotients are constructed through elliptic fibrations with a 2-torsion section,
so that in fact $X$ and $\Km(A)$, interpreted as elliptic curves over $k(t)$, are related by an isogeny.
This will also be our preferred approach in the sequel.
Namely we will construct 3-dimensional families of elliptic fibrations with MW-rank 1 and 2-torsion sections exhibiting 
a Shioda-Inose structure.
As stated in Theorem \ref{thm},
this construction allows us to realise an infinite series of families of K3 surfaces
with sandwich structure
including the five families known before.
For two cases of small $N$ missing from the above list of examples ($N=4,8$)
we will also provide explicit equations for the families.

\section{First series}
\label{s1}

Our starting point is a 4-dimensional family of K3 surfaces 
whose generic member $X$ has 
\[
\NS(X) = U + E_7^2.
\]
This family which has also recently been investigated in \cite{CD},
can be given as an elliptic fibration with two singular fibres of Kodaira type $III^*$ at $0$ and $\infty$:
\begin{eqnarray}
\label{eq:def1}
X: \;\;\; y^2 = x^3 + t^3\,a(t)\,x+t^5 b(t),\;\;\;\; a(t), b(t) \in k[t] \text{ of degree } 2.
\end{eqnarray}
Here one can still rescale $(x,y)$ and separately $t$ to normalise 2 coefficients.
The hyperbolic plane $U\subset\NS(X)$ is spanned by the zero section $O$ and the general fibre $F$
while the $E_7$'s comprise fibre components disjoint from $O$.

The family is a natural starting point since it specialises to Kumar's family with $\NS=U+E_7+E_8$
by setting, for instance, $\deg(a)\leq 1$.
The transcendental lattice of $X$ can be computed with Nikulin's theory of the discriminant form 
(or as a 2-elementary lattice) as
\begin{eqnarray}
\label{eq:T1}
T_X = U^2 + A_1^2.
\end{eqnarray}
Here we consider the dual lattice $L^\vee$ of a non-degenerate even integral lattice $L$.
It gives rise to the discriminant group $L^\vee/L$, a finite abelian group of size the square of the discriminant of $L$.
The discriminant group comes with an induced quadratic form from $L$ which is denoted by 
\[
q_L: L^\vee/L \to \Q/2\Z.
\]
Presently we have isomorphisms of abelian groups
\[
E_7^\vee/E_7\cong A_1^\vee/A_1\cong \Z/2\Z.
\]
In fact there are generators of square $-3/2$ resp.~$-1/2$ which give a direct identification
\begin{eqnarray}
\label{eq:q1}
q_{E_7} \cong -q_{A_1}.
\end{eqnarray}

\subsection{Lattice enhancement}

A convenient way to specialise our 4-dimensional family of K3 surfaces
to a subfamily of Picard number $\rho\geq 17$
is to enhance $L=\NS$ by some vector $v$ of $T_X$ of negative square.
In this context the only subtlety consists in 
the primitive closure in the K3 lattice $\Lambda$:
\[
L':=\overline{\langle L,v\rangle} \subset \Lambda = H^2(X,\Z)=U^3+E_8^2.
\]
Then the theory of lattice polarised K3 surfaces guarantees 
that $L'$ corresponds to a 3-dimensional (sub)family 
consisting of K3 surfaces with generically $\NS=L'$
and transcendental lattice $T'=v^\bot\subset T_X$.
Since we are interested in transcendental lattices containing two copies of $U$ by Theorem \ref{thm:Mo},
we will always enhance $L$ with a vector from $A_1^2$ in the sequel.

\begin{Example}
\label{ex:N=1}
Take the vector $v=(1,0)\in A_1^2$, i.e.~a generator of a single $A_1$.
Via the isomorphism \eqref{eq:q1},
$v/2$ corresponds to a generator $w$ of $E_7^\vee/E_7$.
Thus we obtain 
\[
L':=\overline{\langle L,v\rangle} = \langle L, v/2+w\rangle.
\]
In fact, we have just glued together one copy of $E_7$ and $A_1$ each to the unimodular lattice $E_8$, so
$L'=U+E_7+E_8$.
\end{Example}

Since lattice enhancements involve the primitive closure,
we have to consider which integers $A_1^2$ represents primitively,
i.e.~by a vector $v=(v_1,v_2)$ with gcd$(v_1, v_2)=1$.
An easy exercise in quadratic forms gives

\begin{Lemma}
\label{lem1}
$A_1^2$ represents $-2N$ primitively if and only if
$N$ is a product of primes $\equiv 1\mod 4$ or twice such.
\end{Lemma}

We have already discussed the lattice enhancement for $N=1$ in Example \ref{ex:N=1}.
The other cases correspond to $v_1v_2\neq 0$ and not both even.
Let $w_1, w_2$ denote the generators of $E_7^\vee/E_7$ of square $-3/2$,
matching those for the $A_1$'s via \eqref{eq:q1}.
Then we find 
\[
L':=\overline{\langle L,v'\rangle} \;\; \text{ with } \;\; v'=v/2+\begin{cases}
w_1 & v_1 \text{ odd}\\
0 & v_1 \text{ even}
\end{cases}
+\begin{cases}
w_2 & v_2 \text{ odd}\\
0 & v_2 \text{ even}
\end{cases}
\]
Note that according to our set-up $v'$ is indeed an integral even vector: %by Lemma \ref{lem1}:
\[
2\Z\ni v'^2=-N/2-\begin{cases}
3/2 & N \text{ odd}\\
3 & N \text{ even}
\end{cases}
\]
If $N\neq 1$, then the additional algebraic class $v'$ can directly be identified with a section $P$ 
made orthogonal to $O$ and $F$ in $\NS$ while meeting one or both $III^*$ fibres in the non-identity component depending on the parity of $v_1$ and $v_2$.
In terms of the theory of Mordell-Weil lattices \cite{ShMW},
the contributions from $w_1, w_2$ to $v'^2$
correspond to correction terms for those fibre components.
In consequence $P$ has height $h(P)=N/2$.

\subsection{Transcendental lattice} 
By construction the subfamily with $L'\hookrightarrow \NS$ has generically 
transcendental lattice $T'=U^2+\langle-2N\rangle$.

\subsection{Alternate elliptic fibration}

We shall now switch to an alternative elliptic fibration on $X$ which comes with a 2-torsion section.
For this purpose we identify a divisor $D$ of Kodaira type $I_8^*$ supported on $O$ and the components of the $III^*$ fibres as depicted in the next figure.

\begin{figure}[ht!]
\setlength{\unitlength}{.45in}
\begin{picture}(7,5.2)(-1,-.5)
\thicklines

\multiput(0,3)(1,0){7}{\circle*{.1}}
\put(0,3){\line(1,0){6}}

\multiput(0,1)(1,0){7}{\circle*{.1}}
\put(0,1){\line(1,0){6}}

\put(5,1){\circle{.2}}
\put(5,3){\circle{.2}}

\put(3,4){\circle*{.1}}
\put(3,3){\line(0,1){1}}
\put(3,0){\circle*{.1}}
\put(3,0){\line(0,1){1}}

%\put(6,1.5){\circle*{0.1}}
%\put(5,1){\line(2,1){1}}

%\put(2,1.5){\circle*{.1}}
%\put(1,1){\line(2,1){1}}

%\multiput(-1,2)(0.75,0.75){2}{\circle*{.1}}

%
%\put(-1,5){\line(1,-1){1.5}}
%\put(-1,2){\line(1,1){1.5}}
%\put(.5,3.5){\line(1,0){4}}

%%sections

\put(0,2){\circle{.1}}
\put(0,1){\line(0,1){2}}
\put(0.1,2.1){$O$}

%\multiput(-3,2)(1,0){3}{\circle*{.1}}
%\multiput(-2,3)(0,-2){2}{\circle*{.1}}
%\put(0,2){\line(-1,0){3}}
%\put(-2,1){\line(0,1){2}}

%%%
\thinlines

\put(-0.4,-.35){$u=\infty$}
%\put(-3.4,0.65){$u=0$}

\put(-0.5,-0.5){\framebox(5,5){}}
%\put(-3.5,0.5){\framebox(2,3){}}

\end{picture}
\caption{$I_8^*$ Divisor  supported on two $III^*$'s and $O$}
\label{Fig}
\end{figure}
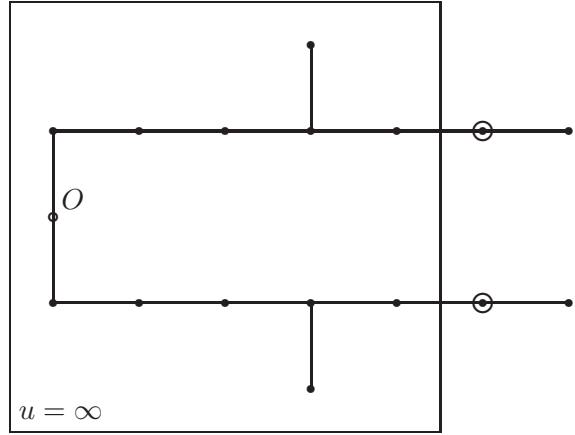

The linear system of this divisor $D$ will induce another  elliptic fibration $X\to\PP^1$.
 Here the rational curves adjacent to $D$ will serve as zero section and 2-torsion section.
 The latter claim is easily verified with the height pairing
 after realising that the two remaining components of the original $III^*$ fibres
 (disjoint from $I_8^*$)
 generically sit in two fibres of type $I_2$.
 Other than these, the new fibration has generically 6 fibres of type $I_1$.
 
 Explicitly the divisor $D$ can be extracted by the parameter $u=x/t^2$ with respect to \eqref{eq:def1}.
 Elementary transformations give the Weierstrass form
 \begin{eqnarray}
 \label{eq:def2}
X:\;\;\;  y^2 = t(u^3t^2+u\,a(t)+b(t)),
 \end{eqnarray}
 an elliptic curve over $k(u)$ with 2-torsion section $(0,0)$.
Upon lattice enhancement,
the rational curve $P$ gives a multisection for the alternate fibration
which induces a section $P'$ again of height $h(P')=N/2$.

\subsection{Kummer surface}

Consider the 2-isogenous elliptic surface $Y$ arising from the alternate elliptic fibration on $X$ by
quotienting by translation by $(0,0)$ and resolving singularities.
Generically $Y$ has singular fibres $I_4^*, 2\times I_1, 6\times I_2$ on top of the 2-torsion section.
The transcendental lattice can be computed as $T_Y = U(2)^2 + A_1^2$.
Now we consider the specialisation $Y'$ of $Y$ corresponding to the lattice enhancement of $X$ yielding the section $P$ (and $P'$).

\begin{Proposition}
\label{prop1}
Let $N$ be an odd integer as in Lemma \ref{lem1}. 
Then $Y'$ is a Kummer surface with $T_{Y'}=T_X(2)=U(2)^2+\langle-4N\rangle$.
\end{Proposition}

\begin{proof}
The property of being a Kummer surface follows 
from the 2-divisibility of $T_{Y'}$,
giving an even lattice of rank at least 17 as in Theorem \ref{thm:Mo}.
We shall now prove that the transcendental lattice takes exactly the shape as stated.

The enhanced section $P'$ pulls back to a section $P^*$ on $Y$ of height $h(P^*)=2h(P')=N$.
We claim that this section is not 2-divisible in $\MW(Y')$.
Otherwise there were a section $Q$ with $2Q=P$, so $h(Q)=N/4$.
But then, the correction terms in the height pairing are all half integers by inspection of the present singular fibres.
Hence $h(Q)\in\frac 12\Z$, giving the required contradiction.

It follows that $\NS(Y')$ is generated by fibre components, zero and 2-torsion sections together with $P^*$.
In particular we compute by \cite[(22)]{SSh}
\begin{eqnarray}
\label{eq:disc1}
\mbox{disc} \NS(Y') = \dfrac{4\times 2^6}{2^2} h(P^*) = 2^6 N.
\end{eqnarray}
Then we use that push-forward induces an embedding
\[
T_X(2) \hookrightarrow T_{Y'}
\]
(which need not be primitive in general, but both lattices have the same rank by \cite{Inose}; see \cite{SI}, \cite{Mo}).
By \eqref{eq:disc1} both lattices have the same discriminant,
hence they agree.
\end{proof}

\begin{Remark}
If $N$ is even, then quite on the contrary the section $P^*$ becomes 2-divisible in $\MW(Y')$
(by inspection of the 2-length, compare Remark \ref{rem}).
Hence the discriminants of $T_X(2)$ and $T_{Y'}$ do not match.
In fact, the two cases of lattice enhancements from Lemma \ref{lem1} are swapped:
odd $N$ on $X$ corresponds to $2N$ on $Y$,
but even $N$ on $X$ gives odd $N/2$ on $Y$.
\end{Remark}

\subsection{Proof of Theorem \ref{thm} (1)}
\label{ss1}

Let $N$ be composed of primes $\equiv 1 \mod 4$ as in Theorem \ref{thm} (1).
Let $X'$ be the family of K3 surfaces 
given by the lattice enhancement of $X$ corresponding to $N$.
Then any $X'$ with $\rho(X')=17$ fits into a sandwiched Shioda--Inose structure by
Proposition \ref{prop1}
using the 2-isogeny and its dual.
Thus we only have to rule out that the construction degenerates for higher Picard number.

Note that any non-degenerate member of the family 
comes with the corresponding elliptic fibration with 2-torsion section (possibly with different singular fibres).
It remains to show that the quotient does always have 2-divisible transcendental lattice,
regardless of the Picard number.
This follows from the next lemma. \qed

\begin{Lemma}
\label{lem0}
Let $(X,\iota)$ be a family of K3 surfaces with Nikulin involution $\iota$.
Assume that the resolution $Y$ of the quotient $X/\iota$ generically has transcendental lattice
\[
T_Y = T_X(2).
\]
Then the same applies to any non-degenerate specialisation of $(X,\iota)$.
\end{Lemma}

\begin{proof}
Let $(X',\iota')$ be a non-degenerate specialisation with quotient $Y'$.
Naturally these come with primitive embeddings
\[
T_{X'} \hookrightarrow T_X,\;\;\;\;\; T_{Y'} \hookrightarrow T_Y.
\]
Using the push-forward via the quotient map and the assumption from the lemma,
we obtain the diagram
$$
\begin{array}{ccc}
T_{Y'} & \hookrightarrow & T_Y\\
\cup && \| \\
T_{X'}(2) & \hookrightarrow & T_X(2)
\end{array}
$$
Going through the lower right corner, we find that $T_{X'}(2)$ embeds primitively into $T_Y$.
Arguing with the upper left corner, 
we deduce the same for the inclusion  $T_{X'}(2)\subset T_Y$.
Since these lattices have the same rank by \cite{Inose}, we deduce equality.
\end{proof}

\subsection{Optimality of construction}

\label{ss:further}

The alert reader might wonder whether 
we might not be able to derive more K3 surfaces from our construction
by not limiting ourselves to the summand $A_1^2$ of $T$ in Lemma \ref{lem1}.
Indeed the transcendental lattice $T$ does represent any integer (in many ways).
However, recall that in essence we aim at primitive embeddings 
\begin{eqnarray}
\label{eq:prim}
U^2+\langle-2N\rangle\hookrightarrow T.
\end{eqnarray}
Here we explain that our construction  does not give rise to any subfamilies with
such transcendental lattice other than those stated in Theorem \ref{thm}.
For this purpose, we first embed the summand $U^2$ into $T$.
In general, the orthogonal complement $M=(U^2)^\bot\subset T$
need not be unique,
but for sure the genus of $M$ is fixed by $T$.
Presently, the genus of $M$ consists of a single lattice
by class group theory.
Hence in order to study primitive embeddings \eqref{eq:prim}
 it does indeed suffice to consider representations of $-2N$ by  $A_1^2$.
Verbatim, the same argument will go through for the other two constructions 
(see Lemma \ref{lem2}, \ref{lem3}).

\subsection{Other lattice enhancements}

For the sake of completeness,
we briefly comment on other lattice enhancements of the family $X$.
This will also serve as a  sanity check of the above argument.
Analogous arguments apply to the K3 families in the next two sections.

Let $X'$ be a family of K3 surfaces of Picard number $17$ and discriminant $2N (N>1)$
that  arises from $X$ by lattice enhancement.
As before, $X'$ comes with an additional section $P$ of height $N/2$,
say with respect to the original elliptic vibration.
Consider the case where $N$ is odd.
Then $P$ meets some non-identity fibre components,
but the section $2P$ of height $2N$ does not.
Let $\varphi$ denote the orthogonal projection with respect to $O,F$ in $\NS(X')$.
By assumption $\varphi(2P)$ is orthogonal to the trivial lattice
(i.e.~to the image of $\NS(X)$ in $\NS(X')$), and we have
\[
\varphi(2P)^2 = -2N,\;\;\; \varphi(2P).\varphi(P) = -N.
\]
In particular we find that $\frac 1N\varphi(2P)\in \NS(X')^\vee$.
In the discriminant group $\NS(X')^\vee/\NS(X')$
this induces an element of order $N$
where the discriminant form evaluates as $-2/N$.

Now assume that $T(X')=U^2+\langle-2N\rangle$ is one of the transcendental lattices in question.
Its discriminant form also evaluates as $-2/N$ at an element of order $N$.
But then for $\NS(X')$ and $T(X')$ to be orthogonal complements in the K3 lattice,
we require that their discriminant forms have reversed signs.
In particular this implies
that $-1$ is a square modulo any prime divisor of $N$.
Thus we find exactly the conditions of Theorem \ref{thm}.

For $N$ even, we distinguish two further cases by the parity of $M=N/2$.
If $N/2$ is odd, then essentially the same argument as above goes through
with the element  $\frac 1M\varphi(P)\in\NS(X')^\vee$ of order $N$ in the discriminant group.
If $N/2$ is even, then $P$ is in the narrow Mordell-Weil lattice, meeting all singular fibres at the identity component.
Thus $\NS(X') = \NS(X) + \langle\varphi(P)\rangle$,
and the discriminant group has 2-length 3 exceeding that of $U^2+\langle-2N\rangle$,
contradiction.

\section{Second series}
\label{s2}

In this and the next section,
we shall argue directly with elliptic fibrations with 2-torsion section,
this time semi-stable.
Such fibrations admit an extended Weierstrass form
\[
y^2 = x(x^2+a(t)x+b(t))
\]
with reducible fibres at the zeros of $b(t)$.
We treat two families which lend themselves directly to Shioda--Inose stuctures.
In this section, we consider surfaces with singular fibres of Kodaira type $I_{14}, I_2$.
The corresponding K3 surfaces form a 4-dimensional family
\begin{eqnarray}
\label{eq14}
X:\;\;\; 
y^2 = x(x^2+a(t)x+t)
\end{eqnarray}
where $a\in k[t]$ has degree 4 and does not vanish at $t=0$.
Generically $\rho(X)=16$.

\subsection{Transcendental lattice}

The transcendental lattice can be read off directly from an alternative elliptic fibration on $X$.
Namely it is easy to extract a divisor of Kodaira type $II^*$ from $I_{14}$ extended by $O$ and the identity component of $I_2$. 
The remaining components of $I_{14}$ give rise to a section and an $A_6$ configuration of rational curves.
Comparing ranks (and discriminants),
we find generically
\[
\NS(X) = U+A_6+E_8 \;\;\text{ and } \;\;\; T_X = U^2 + \begin{pmatrix}-2 & -1\\-1 & -4\end{pmatrix}.
\]
The latter representation is easily verified with the discriminant form
since $A_6^\vee/A_6$ has a generator of square $-6/7$.

\subsection{Quotient family}

Consider the quotient by translation by  the 2-torsion section and 
denote the resulting elliptic K3 surfaces by $Y$.
Then generically $Y$ has singular fibres $I_7, I_1, 8\times I_2$ and $\MW\cong \Z/2\Z$.

\begin{Lemma}
\label{lem2}
Generically $Y$ has transcendental lattice $T_Y\cong T_X(2)$.
\end{Lemma}

\begin{proof}
By \cite{Inose} $Y$ also has Picard number 16,
so the MW-rank is 0 by the Shioda-Tate formula.
Standard formulas exclude any further torsion.
Hence $\NS(Y)$ is generated by fibre components and the two torsion sections;
in particular $\NS(Y)$ has discriminant $2^67$.
Then, as before, the push-forward embedding $T_X(2)\hookrightarrow T_Y$ is an isometry.
\end{proof}

\subsection{Lattice enhancement}

As before we enhance $\NS(X)$ by a vector from the last summand of $T_X$.
Generally we find the following possibilities:

\begin{Lemma}
\label{lem2}
$ \begin{pmatrix}-2 & -1\\-1 & -4\end{pmatrix}$ represents $-2N$ primitively if and only if
$N$ is a product of primes $\equiv 1,2,4\mod 7$ or seven times such a product.
\end{Lemma}

Now pick a vector $v$ as in Lemma \ref{lem2}
and enhance $\NS$ by a generator of $v^\bot$ in the above rank 2 lattice.
We infer the transcendental lattice 
\[
T'=U^2+\Z v = U^2+\langle -2N\rangle.
\]
On the elliptic fibrations this implies an additional singular fibre ($N=1$) or a section $P$ of height $2N/7$ ($N>1$).

\subsection{Proof of Theorem \ref{thm} (2)}
\label{ss2}

Let $N$ be  as in Theorem \ref{thm} (2) (or equivalently Lemma \ref{lem2}).
Let $X'$ be a member of the subfamily of K3 surfaces 
given by the lattice enhancement of $X$ corresponding to $N$ as above.
Then by assumption $T_{X'}$ embeds primitively into $U^2+\langle-2N\rangle$.
Moreover the quotient surface has transcendental lattice $T_{X'}(2)$ by Lemma \ref{lem0}
in conjunction with Lemma \ref{lem2}.
Using the 2-isogeny and its dual,
 these surface realise a sandwiched Shioda--Inose structure. \qed

\begin{Remark}
\label{rem}
The above construction implies that the image of the section $P$ on the quotient surface
becomes 2-divisible. 
This fact can also be checked directly on the quotient surface
by an argument comparing the 2-length of $\NS$ to the rank of the transcendental lattice.
\end{Remark}

%Then any $X_N$ with $\rho(X_N)=17$ fits into a sandwiched Shioda--Inose structure by
%Proposition \ref{prop1}
%using the 2-isogeny and its dual.
%Thus we only have to rule out that the construction degenerates for higher Picard number.

\subsection{Example: $N=4$}

We need to endow the elliptic fibration \eqref{eq14} with a section of height $8/7$.
Up to translation by the 2-torsion section and inversion, this is uniquely achieved by a section $P$ meeting $I_{14}$ at $\Theta_4$ (the fourth component) and $I_2$ at the identity component, but not intersecting $O$.
Indeed the height pairing gives $h(P)=4-\frac{4\cdot10}{14}=\frac 87$.

With the present extended Weierstrass form \eqref{eq14},
$P$ can only take the shape $(\alpha, w)$ with $\alpha\in k^*$ and $w\in k[t]$ of degree 2.
But then the pair $(\alpha, w)$ uniquely determines the polynomial $a(t)$ in \eqref{eq14} by
\[
w^2=\alpha^3+\alpha^2 a(t) + \alpha t.
\]
Hence we obtain the 3-dimensional family of K3 surfaces 
with generically $T=U^2+\langle-8\rangle$
(and the unirationality of the corresponding moduli space).

\section{Third series}
\label{s3}

As third series we treat elliptic K3 surfaces with 2-torsion section 
and singular fibres of type $I_{10}, I_6$.
Here the Weierstrass form can be transformed to
\begin{eqnarray}
\label{eq10}
X:\;\;\; 
y^2 = x(x^2+a(t)x+t^3),
\end{eqnarray}
again yielding a 4-dimensional family with $\rho=16$ generically.
The arguments are very similar to the previous family,
so we only outline the main ideas.

\subsection{Transcendental lattice}

With the discriminant form one finds the transcendental lattice generically as
\[
T_X = U^2 + \begin{pmatrix}-4 & -1\\-1 & -4\end{pmatrix}.
\]
The computations may be eased by switching to another elliptic fibration,
for instance with fibres of type $III^*$ and $IV^*$ and a section of height $5/2$
(elliptic parameter $u=x/t$).
This gives a representation of $q_{\NS(X)}$ as $\Z/3\Z(-4/3) + \Z/5\Z(-2/5)$
which agrees with $q_{T_X}$ up to sign.

\subsection{Quotient family}

Denote the quotient elliptic surfaces with the 2-torsion section by $Y$.
Then generically $Y$ has singular fibres $I_5, I_3, 8\times I_2$ and $\MW\cong \Z/2\Z$.

\begin{Lemma} 
\label{lem3}
Generically $Y$ has transcendental lattice $T_Y\cong T_X(2)$.
\end{Lemma}

\begin{proof}
By inspection of the singular fibres and torsion section  $\NS(Y)$ has discriminant $2^615$.
Then, as before, the push-forward embedding $T_X(2)\hookrightarrow T_Y$ is an isometry.
\end{proof}

\subsection{Lattice enhancement}

We continue by enhancing $\NS(X)$ by a vector from the last summand of $T_X$.
The analysis of the possible numbers which are represented primitively is a little more delicate.
One can already observe this from the fact that each $2, 3$ and $5$ is represented,
but neither $6, 10$ or $15$.
This is related to the fact that the corresponding quadratic form
gives the 2-torsion class in the class group $Cl(-15)$.
This explains why a parity condition enters for the representations:

\begin{Lemma}
\label{lem3}
An integer $-2N$ is represented primitively by $\begin{pmatrix}-4 & -1\\-1 & -4\end{pmatrix}$ if and only if
\begin{itemize}
\item
$N$ is a product of an odd number of primes $\equiv 2,8\mod 15$ or $15$ times such a product;
\item
$N$ is 3 or 5 times a product of an even number of primes $\equiv 2,8\mod 15$.
\end{itemize}
\end{Lemma}

Enhancing $\NS$ by a primitive vector perpendicular to the vector from the lemma
gives the transcendental lattice 
\[
T'=U^2+\langle -2N\rangle.
\]
The elliptic fibration is thus endowed with a section $P$ of height $2N/15$.

\subsection{Proof of Theorem \ref{thm} (3)}
\label{ss3}

Let $N$ be  as in Theorem \ref{thm} (3) (or equivalently Lemma \ref{lem3}).
Let $X'$ be a member of the subfamily of K3 surfaces 
given by the lattice enhancement of $X$ corresponding to $N$ as above.
Then by assumption $T_{X'}$ embeds primitively into $U^2+\langle-2N\rangle$.
Moreover the quotient surface has transcendental lattice $T_{X'}(2)$ by Lemma \ref{lem0}
in conjunction with Lemma \ref{lem3}.
Thus these surface realise a sandwiched Shioda--Inose structure. \qed

%\begin{Remark}
%The above construction implies that the image of the section $P$ on the quotient surface
%becomes 2-divisible. 
%This fact can also be checked directly on the quotient surface
%by an argument comparing the 2-length of $\NS$ to the rank of the transcendental lattice.
%\end{Remark}

%Then any $X_N$ with $\rho(X_N)=17$ fits into a sandwiched Shioda--Inose structure by
%Proposition \ref{prop1}
%using the 2-isogeny and its dual.
%Thus we only have to rule out that the construction degenerates for higher Picard number.

\subsection{Example: $N=8$}

The elliptic fibration \eqref{eq10} ought to be equipped with a section $P$ of height $16/15$.
By the height pairing it suffices that $P$ meets both $I_{10}$ and $I_6$
 at $\Theta_2$ (the second component) while not intersecting $O$ as
$h(P)=4-\frac{2\cdot8}{10}-\frac{2\cdot4}6=\frac{16}{15}$.

The present Weierstrass form \eqref{eq10} implies the shape of
$P$ to be $(\alpha t^2, wt^2)$ with $\alpha\in k^*$ and $w\in k[t]$ of degree 2.
But then the pair $(\alpha, w)$ uniquely determines the polynomial $a(t)$ in \eqref{eq10}
by
\[
w^2=\alpha^3t^2+\alpha^2 a(t) + \alpha t.
\]
Hence we obtain the 3-dimensional family of K3 surfaces 
with generically $T=U^2+\langle-16\rangle$
(and the unirationality of the corresponding moduli space).

\subsection*{Acknowledgements}
The starting points to this paper came up during the 
"Workshop on Arithmetic and Geometry of K3 surfaces and Calabi-Yau threefolds"
at the Fields Institute in August 2011.
Thanks for the stimulating atmosphere and the great hospitality.

\end{document}